\documentclass{article}

\usepackage{iclr2026_conference,times}
\usepackage{amsmath,amssymb,amsthm,mathtools,bm}
\usepackage{booktabs,array,multirow}
\usepackage{graphicx}
\usepackage{microtype}
\usepackage{enumitem}
\usepackage{xcolor}
\usepackage{hyperref}
\usepackage[nameinlink,capitalize,noabbrev]{cleveref}

\hypersetup{
  colorlinks=true,
  linkcolor=blue!55!black,
  citecolor=blue!55!black,
  urlcolor=blue!55!black,
  pdftitle={Forward Tangent-Centered Sinkhorn DRO: Exact Nominal Recovery from Optimal Transport to KL},
  pdfauthor={Junjie Guo}
}

\newtheorem{theorem}{Theorem}[section]
\newtheorem{proposition}[theorem]{Proposition}
\newtheorem{lemma}[theorem]{Lemma}

\theoremstyle{definition}

\theoremstyle{remark}

\newcommand{\R}{\mathbb{R}}
\newcommand{\E}{\mathbb{E}}

\newcommand{\one}{\mathbf{1}}
\newcommand{\simplex}{\Delta_m}
\newcommand{\cU}{\mathcal{U}}
\newcommand{\cT}{\mathcal{T}}
\newcommand{\TC}{\mathrm{TC}}
\newcommand{\KL}{\mathrm{KL}}
\newcommand{\diag}{\operatorname{Diag}}

\newcommand{\inner}[2]{\left\langle #1,#2\right\rangle}
\newcommand{\norm}[1]{\left\lVert #1\right\rVert}
\newcommand{\abs}[1]{\left\lvert #1\right\rvert}
\newcommand{\trans}{^{\mathsf T}}

\title{Forward Tangent-Centered Sinkhorn DRO:\\Exact Nominal Recovery from Optimal Transport to KL}

\author{Junjie Guo\\
Rutgers University\\
\texttt{jg1806@scarletmail.rutgers.edu}}

\iclrfinalcopy

\begin{document}
\maketitle
\lhead{}
\renewcommand{\headrulewidth}{0pt}

\begin{abstract}
Fixed-reference entropy-regularized transport yields smooth distributionally robust optimization (DRO) objectives, but its value as a function of the adversarial marginal is generally minimized at a kernel-smoothed law rather than at the nominal distribution. Consequently, the smallest raw ambiguity set need not recover nominal risk. Building on prior information-geometric plan-KL orientations, we study the \emph{forward tangent-centered} discrepancy obtained by subtracting the value and first variation at the nominal law. This is the unique affine recentering, up to the simplex gauge; for symmetric costs, its plan-KL representation coincides with the forward candidate noted by \citet{amari2018information}. The representation shows that the discrepancy dominates marginal KL, is uniquely zero at the nominal law, and is strongly convex in total variation. The resulting DRO problem retains a scalar log-sum-exp dual and closed-form conditional adversarial weights. On finite supports with symmetric zero-diagonal costs, the discrepancy converges uniformly to the unregularized transport cost at low temperature; for any finite cost, its rescaled value converges uniformly to marginal KL at high temperature. We derive its exact local Hessian, an explicit finite-sample containment radius, and a generalized chi-square calibration for multinomial data. Controlled simulations verify nominal centering, both limiting regimes, and practical plug-in coverage.
\end{abstract}

\section{Introduction}

Entropy regularization makes optimal transport (OT) smooth and computationally accessible, and it is now a standard ingredient in transport-based learning and DRO \citep{cuturi2013sinkhorn,peyre2019computational,wang2025sinkhorn}. Yet entropy has two distinct effects when one transport marginal is adversarial: it smooths the optimization problem, and it selects a baseline distribution through the reference measure. For fixed-reference Sinkhorn DRO, \citet[Remark~5]{wang2025sinkhorn} identify the zero-budget law as a kernel smoothing of the nominal distribution. Thus a set described as being ``centered'' at a nominal law $p$ can have its minimum at another law.

The issue is already visible on a finite support. For a cost matrix $C$ and temperature $\gamma>0$, let
\begin{equation}\label{eq:intro-T}
T_\gamma(q,p)=\min_{\Pi\one=q,\,\Pi\trans\one=p}
\left\{\inner{C}{\Pi}+\gamma\sum_{j,i}\Pi_{ji}(\log\Pi_{ji}-1)\right\}.
\end{equation}
If $C=0$, the minimizer over $q$ is uniform, regardless of $p$. Merely shifting the radius to make $p$ feasible does not remove this first-order preference. The natural question is therefore:
\begin{quote}
Can one remove the fixed-reference drift while preserving the convexity and log-sum-exp structure that make Sinkhorn DRO useful?
\end{quote}

We answer affirmatively by subtracting the tangent plane of $q\mapsto T_\gamma(q,p)$ at $p$, using the forward Bregman orientation dictated by the adversarial marginal. If $\Pi_\gamma^{q,p}$ is the entropic optimal plan and $\Pi_\gamma^{p,p}$ is the nominal self-plan, then
\begin{equation}\label{eq:intro-plan-kl}
D_\gamma^{\TC}(q\Vert p)
=\gamma\KL\!\left(\Pi_\gamma^{q,p}\middle\Vert\Pi_\gamma^{p,p}\right)
\ge \gamma\KL(q\Vert p).
\end{equation}
The correction therefore has an information-geometric interpretation at the level of transport plans, controls probability reweighting more sharply than Pinsker's inequality alone, and retains the original cost geometry through the optimal-plan map.

\paragraph{Contributions.}
Our contributions are threefold.
\begin{enumerate}[leftmargin=1.45em,itemsep=2pt,topsep=3pt]
\item For the forward counting-reference orientation, we characterize the raw center, prove that tangent subtraction is the unique affine correction restoring $p$, connect its symmetric-cost plan-KL form to the forward candidate of \citet{amari2018information}, and derive its KL chain-rule decomposition.
\item We derive a strong finite-dimensional DRO dual with a single scalar transport multiplier and explicit conditional adversarial weights. For symmetric zero-diagonal discrete costs, we prove a uniform OT limit as $\gamma\downarrow0$; for any finite cost, we prove a uniform KL limit for $D_\gamma^{\TC}/\gamma$ as $\gamma\to\infty$.
\item We identify the local metric through a conditional-softmax covariance matrix. This yields an explicit nonasymptotic multinomial radius and an asymptotically exact generalized chi-square calibration; controlled simulations evaluate the latter plug-in procedure.
\end{enumerate}

The analysis is intentionally finite-dimensional. It covers categorical laws, scenario reduction, predictive or importance-weighted supports, and other settings in which nominal probabilities are nonuniform. It does \emph{not} provide a continuous-support sampling theorem or account for a learned support or ground cost.

\paragraph{Closest work.}
\citet[Eqs.~49--50]{amari2018information} introduce both KL orientations between optimal plans and analyze the self-plan-first orientation. For symmetric costs, their dual candidate with the nominal reference, after transposing plans and matching regularization conventions, coincides with the forward plan-KL quantity in \eqref{eq:intro-plan-kl}. We therefore do not claim that candidate itself as new. Our contribution is to identify it as the unique tangent recentering of fixed/counting-reference transport and develop the consequences needed for DRO. Bregman divergences generated by one-marginal product-reference entropic OT have also appeared as mirror maps \citep{kerimkulov2025mirror}; the exact relationship is given in \eqref{eq:product-decomposition}. General Sinkhorn-DRO duality is due to \citet{wang2025sinkhorn}, and regularized Wasserstein DRO is studied by \citet{azizian2023regularization,azizian2023exact}.

Symmetric Sinkhorn divergences instead target two-sample comparison and interpolate between OT and a kernel discrepancy \citep{feydy2019interpolating,genevay2019sample,sejourne2019sinkhorn,aubinfrankowski2026debiasing}. The $n^{-1}$ scale and generalized chi-square mechanism are known for finite-space Sinkhorn losses \citep{bigot2019central,goldfeld2024limit}, and explicit pseudoinverse geometry is available for the symmetric divergence \citep{lavenant2026riemannian}. Our distinctive results are the forward fixed-reference Sinkhorn-DRO reformulation and worst-case law, its assumption-explicit low-temperature OT and general high-temperature KL limits, and the compact conditional-softmax Hessian used for one-sided ambiguity-radius calibration. See the appendix for a precise algebraic comparison. Regularized OT's broader Bregman geometry is surveyed by \citet{dessein2018regularized}.

\section{Fixed-reference drift and canonical recentering}

\subsection{Setup and the raw center}

Let
$\simplex=\{q\in\R_+^m:\one\trans q=1\}$ and
$\simplex^\circ=\{q\in\simplex:q_j>0\}$.
Throughout, $p\in\simplex^\circ$, $C\in\R_+^{m\times m}$ is finite, and $0\log0=0$. The coupling polytope is
$\Gamma(q,p)=\{\Pi\ge0:\Pi\one=q,\Pi\trans\one=p\}$,
and $T_\gamma$ is defined by \eqref{eq:intro-T}. This is entropy relative to counting measure, up to a constant. Multiplying a fixed reference matrix by positive row or column factors only adds affine terms in $q$ and constants in $p$; \cref{prop:reference-invariance} below shows that tangent centering removes them. A reference that itself depends on the varying $q$, such as $qp\trans$, defines a different geometry.

The raw minimizer is explicit.
\begin{proposition}[Raw fixed-reference center]\label{prop:raw-center}
For $K_{ji}=e^{-C_{ji}/\gamma}$ and $Z_i=\sum_kK_{ki}$,
\begin{equation}\label{eq:raw-center}
\Pi_{ji}^{\circ}=p_i\frac{K_{ji}}{Z_i},
\qquad
q_{\gamma,j}^{\circ}(p)=\sum_i p_i\frac{K_{ji}}{Z_i}
\end{equation}
uniquely minimize $T_\gamma(q,p)$ over $\Pi$ and $q$, respectively. Hence $q_\gamma^\circ(p)\to m^{-1}\one$ as $\gamma\to\infty$. If $C_{ii}=0$ and $\Delta_C=\min_{i\ne j}C_{ji}>0$, then
$\norm{q_\gamma^\circ(p)-p}_1\le2(m-1)e^{-\Delta_C/\gamma}$.
\end{proposition}
The formula is the discrete kernel smoothing already implicit in the zero-budget representation of \citet{wang2025sinkhorn}. Our goal is not to claim this smoothing as new, but to remove it canonically without sacrificing tractability.

\subsection{Tangent centering and a plan-level identity}

For fixed $p$, write $F_{\gamma,p}(q)=T_\gamma(q,p)$. This function is strictly convex on $\simplex$ and smooth on $\simplex^\circ$. Its relative gradient is defined up to a multiple of $\one$. Choose any representative $\nu_{\gamma,p}\in\partial F_{\gamma,p}(p)$ and define
\begin{equation}\label{eq:tc-def}
D_\gamma^{\TC}(q\Vert p)
:=F_{\gamma,p}(q)-F_{\gamma,p}(p)
-\nu_{\gamma,p}\trans(q-p).
\end{equation}
Thus $D_\gamma^{\TC}$ is an \emph{anchored} Bregman remainder: the convex generator $F_{\gamma,p}$ depends on the anchor $p$.

\begin{proposition}[Canonical affine correction]\label{prop:canonical}
For $a\in\R^m$, let
$D_a(q,p)=F_{\gamma,p}(q)-F_{\gamma,p}(p)-a\trans(q-p)$.
Then $p$ minimizes $D_a(\cdot,p)$ on $\simplex$ if and only if
$a-\nu_{\gamma,p}\in\operatorname{span}\{\one\}$.
Whenever this holds, strict convexity makes $p$ the unique minimizer. Hence \eqref{eq:tc-def} is the unique affine recentering on the simplex, up to gauge.
\end{proposition}

The next theorem is the main structural result. For distributions $P,Q$ on a finite set, we use
$\KL(P\Vert Q)=\sum_xP_x\log(P_x/Q_x)$.
\begin{theorem}[Exact plan-KL geometry]\label{thm:plan-kl}
Let $\Pi_\gamma^{q,p}$ be the unique optimizer defining $T_\gamma(q,p)$. For every $q\in\simplex$,
\begin{align}
D_\gamma^{\TC}(q\Vert p)
&=\gamma\KL\!\left(\Pi_\gamma^{q,p}\middle\Vert\Pi_\gamma^{p,p}\right)\label{eq:plan-kl}\\
&=\gamma\KL(q\Vert p)
+\gamma\sum_{j:q_j>0}q_j
\KL\!\left(\Pi_\gamma^{q,p}(\cdot\mid j)
\middle\Vert\Pi_\gamma^{p,p}(\cdot\mid j)\right).\label{eq:chain-kl}
\end{align}
More generally, for $r\in\simplex^\circ$, the Bregman remainder of the fixed generator $F_{\gamma,p}$ satisfies
\begin{equation}\label{eq:global-plan-kl}
B_{F_{\gamma,p}}(q,r)
=\gamma\KL(\Pi_\gamma^{q,p}\Vert\Pi_\gamma^{r,p})
\ge\gamma\KL(q\Vert r)
\ge\frac{\gamma}{2}\norm{q-r}_1^2.
\end{equation}
Consequently, $D_\gamma^{\TC}(\cdot\Vert p)$ is convex, uniquely zero at $p$, and
\begin{equation}\label{eq:kl-tv-lower}
D_\gamma^{\TC}(q\Vert p)
\ge\gamma\KL(q\Vert p)
\ge\frac{\gamma}{2}\norm{q-p}_1^2.
\end{equation}
Thus, for $\rho\ge0$, $\cU_{\rho,\gamma}^{\TC}(p):=\{q:D_\gamma^{\TC}(q\Vert p)\le\rho\}$ is nonempty, compact, and convex, with
$\cU_{0,\gamma}^{\TC}(p)=\{p\}$ and
$\cU_{\rho,\gamma}^{\TC}(p)\subseteq\{q:\KL(q\Vert p)\le\rho/\gamma\}$.
\end{theorem}

Identity \eqref{eq:plan-kl} says that tangent centering replaces an arbitrary additive convention by a genuine relative entropy between two \emph{optimized} couplings. The residual term in \eqref{eq:chain-kl} measures how much the nominal scenario conditioned on an adversarial state changes beyond marginal reweighting.

This orientation can also be related exactly to prior product-reference EOT Bregman geometry. Define
$h_{\gamma,p}(q)=\min_{\Pi\in\Gamma(q,p)}\{\inner{C}{\Pi}+\gamma\KL(\Pi\Vert qp\trans)\}$
and let $B_{h_{\gamma,p}}(q,p)$ be its forward Bregman remainder. Marginal identities give
\begin{equation}\label{eq:product-decomposition}
D_\gamma^{\TC}(q\Vert p)
=B_{h_{\gamma,p}}(q,p)+\gamma\KL(q\Vert p).
\end{equation}
Thus the conditional term in \eqref{eq:chain-kl} is precisely the product-reference remainder, while counting-reference entropy supplies the marginal KL and strong convexity. For symmetric costs, \eqref{eq:plan-kl} is the dual candidate recorded in \citet[Eq.~50]{amari2018information}; our results identify its tangent-centering origin and develop its DRO geometry. Equation~\eqref{eq:product-decomposition} likewise distinguishes it from the general product-reference mirror map of \citet{kerimkulov2025mirror}.

\begin{proposition}[Fixed-reference invariance]\label{prop:reference-invariance}
Let $R\in\R_{++}^{m\times m}$ be a fixed entropy reference and let
$R'=\diag(r)R\diag(c)$ for $r,c>0$. The tangent-centered discrepancies generated by $R$ and $R'$ coincide. Equivalently, adding arbitrary row and column potentials to $C$ leaves $D_\gamma^{\TC}$ unchanged.
\end{proposition}

\section{A tractable distributionally robust objective}

Let $L(\theta)=(L_1(\theta),\ldots,L_m(\theta))\trans$ and define
\begin{equation}\label{eq:robust-risk}
\Phi_{\rho,\gamma}^{\TC}(\theta)
=\sup_{q\in\cU_{\rho,\gamma}^{\TC}(p)}q\trans L(\theta).
\end{equation}
Normalize the gauge of $\nu=\nu_{\gamma,p}$ so that
$p\trans\nu=T_\gamma(p,p)$, and set
$\bar C_{ji}=C_{ji}-\nu_j$.
Then the nominal self-plan has corrected budget zero, and \eqref{eq:robust-risk} is equivalently a maximization over $\Pi\trans\one=p$ subject to
$\inner{\bar C}{\Pi}+\gamma\sum_{j,i}\Pi_{ji}(\log\Pi_{ji}-1)\le\rho$.

\begin{theorem}[Scalar strong dual]\label{thm:dual}
Suppose $\rho>0$ and $L_j(\theta)$ is finite. Define
$Z_i(\theta,\lambda):=p_i^{-1}\sum_{j=1}^m
\exp\!\left([L_j(\theta)/\lambda-\bar C_{ji}]/\gamma\right)$. Then
\begin{equation}\label{eq:dual}
\Phi_{\rho,\gamma}^{\TC}(\theta)
=\inf_{\lambda>0}\lambda\!\left[
\rho+\gamma+\gamma\sum_{i=1}^mp_i\log Z_i(\theta,\lambda)\right].
\end{equation}
If the infimum is attained at $\lambda^*>0$, the unique worst-case coupling is
\begin{equation}\label{eq:worst-plan}
\begin{aligned}
\Pi_{ji}^*&=p_i
\frac{\exp\!\left((L_j(\theta)-\lambda^*\bar C_{ji})/(\lambda^*\gamma)\right)}
{\sum_k\exp\!\left((L_k(\theta)-\lambda^*\bar C_{ki})/(\lambda^*\gamma)\right)},\\[-2pt]
q_j^*&=\sum_i\Pi_{ji}^*.
\end{aligned}
\end{equation}
If each $L_j$ is convex in $\theta$, the joint objective before minimizing over $\lambda$ is convex in $(\theta,\lambda)$. If the losses are differentiable and $q^*$ is unique, then
$\nabla_\theta\Phi_{\rho,\gamma}^{\TC}(\theta)=\sum_jq_j^*\nabla L_j(\theta)$.
\end{theorem}

The self-potential is computed once by Sinkhorn scaling. Each subsequent dual evaluation costs $O(m^2)$ for a dense kernel and uses stable columnwise log-sum-exp operations. Formula \eqref{eq:worst-plan} is conditional on the scalar multiplier; the $\lambda\downarrow0$ boundary has the usual hard-max interpretation. At $\rho=0$, \cref{thm:plan-kl} gives the exact identity
$\Phi_{0,\gamma}^{\TC}(\theta)=p\trans L(\theta)$ for every $\gamma>0$.

\begin{table}[t]
\caption{Three ways to treat fixed-reference entropic bias. ``Simple dual'' means the raw scalar log-sum-exp form, possibly after a cost shift.}
\label{tab:comparison}
\centering
\small
\begin{tabular}{@{}>{\raggedright\arraybackslash}p{1.15in}ccc@{}}
\toprule
 & Raw $T_\gamma$ & Symmetric Sinkhorn & Tangent-centered \\
\midrule
Nominal is zero/minimizer & generally no & yes$^\dagger$ & yes \\
One-sided convexity & yes & conditional$^\dagger$ & yes \\
Simple DRO dual & yes & generally no & yes \\
High-$\gamma$ geometry & reference-driven & kernel/MMD$^\dagger$ & KL after scaling \\
\bottomrule
\end{tabular}
\par\vspace{2pt}
\parbox{0.97\linewidth}{\footnotesize $^\dagger$\,Under the kernel/cost assumptions of \citet{feydy2019interpolating,aubinfrankowski2026debiasing}.}
\end{table}

\section{A continuum from discrete OT to KL}

Let
$W_C(q,p)=\min_{\Pi\in\Gamma(q,p)}\inner{C}{\Pi}$.
We call this a discrete OT cost; it is a Wasserstein cost when $C$ is induced by an appropriate ground metric. The two temperature limits are uniform, which is important because the adversary also varies with $\gamma$.

\begin{theorem}[Uniform limiting geometry]\label{thm:limits}
Assume first that $C=C\trans$, $C_{ii}=0$, and $p_i\ge p_{\min}>0$.
Set $\varepsilon_\gamma:=2\gamma[1+2\log m+\log(\sqrt m/p_{\min})]$.
Then, for every $\gamma>0$,
\begin{equation}\label{eq:low-bound}
\sup_{q\in\simplex}
\abs{D_\gamma^{\TC}(q\Vert p)-W_C(q,p)}
\le \varepsilon_\gamma.
\end{equation}
For any finite $C$ and $p\in\simplex^\circ$,
\begin{equation}\label{eq:high-limit}
\sup_{q\in\simplex}
\abs{\gamma^{-1}D_\gamma^{\TC}(q\Vert p)-\KL(q\Vert p)}
\longrightarrow0
\qquad(\gamma\to\infty).
\end{equation}
\end{theorem}

The scaling in \eqref{eq:high-limit} is essential. At fixed $\rho$, \eqref{eq:kl-tv-lower} implies that $\cU_{\rho,\gamma}^{\TC}(p)$ collapses to $\{p\}$ as $\gamma\to\infty$. A nondegenerate KL limit uses $\rho=\gamma\tau$: for every $\tau>0$, the sets
$\cU_{\gamma\tau,\gamma}^{\TC}(p)$ converge in Hausdorff distance to
$\{q:\KL(q\Vert p)\le\tau\}$.

The uniform bounds also transfer directly to robust risks. If
$\sup_\theta\norm{L(\theta)}_\infty\le B$ and $\varepsilon_\gamma<\rho$, then
\begin{equation}\label{eq:risk-low-bound}
\sup_\theta\abs{\Phi_{\rho,\gamma}^{\TC}(\theta)
-\sup_{q:W_C(q,p)\le\rho}q\trans L(\theta)}
\le\frac{2B\varepsilon_\gamma}{\rho}.
\end{equation}
Likewise, at radius $\gamma\tau$, bounded-loss robust risks converge uniformly to KL-DRO risks. Thus $\gamma$ controls a genuine modeling continuum, not merely numerical smoothing.

\section{Local geometry and statistical calibration}

The semi-dual log-partition function is
\begin{equation}\label{eq:A}
A_{\gamma,p}(u)=\gamma\sum_i p_i
\log\sum_j\exp((u_j-C_{ji})/\gamma),
\quad
s_{ji}(u)=\frac{e^{(u_j-C_{ji})/\gamma}}
{\sum_ke^{(u_k-C_{ki})/\gamma}}.
\end{equation}
The self-potential satisfies $\sum_i p_i s_i(\nu_{\gamma,p})=p$. Define
\begin{equation}\label{eq:M}
M_{\gamma,p}=\sum_i p_i
\left[\diag(s_i(\nu_{\gamma,p}))
-s_i(\nu_{\gamma,p})s_i(\nu_{\gamma,p})\trans\right],
\qquad
\cT=\{v:\one\trans v=0\}.
\end{equation}

\begin{theorem}[Local metric]\label{thm:hessian}
$M_{\gamma,p}$ is positive definite on $\cT$, and the intrinsic Hessian of $q\mapsto T_\gamma(q,p)$ at $q=p$ is
\begin{equation}\label{eq:H}
H_{\gamma,p}=\gamma M_{\gamma,p}^{\dagger}.
\end{equation}
Here $\dagger$ denotes the Moore--Penrose pseudoinverse, equivalently inversion on $\cT$.
Consequently,
$D_\gamma^{\TC}(q\Vert p)=\tfrac12(q-p)\trans H_{\gamma,p}(q-p)+o(\norm{q-p}_2^2)$.
Moreover, for every $v\in\cT$,
\begin{equation}\label{eq:fisher-lower}
v\trans H_{\gamma,p}v\ge\gamma\sum_j\frac{v_j^2}{p_j}.
\end{equation}
\end{theorem}
The lower bound is the local counterpart of $D_\gamma^{\TC}\ge\gamma\KL$ and follows also from the law of total covariance.

Now suppose $p^\star\in\simplex^\circ$ generates
$N\sim\operatorname{Multinomial}(n,p^\star)$. To keep the nominal law interior, use the add-$a$ estimator
$\widetilde p_{n,j}=(N_j+a)/(n+ma)$, $a>0$.

\begin{theorem}[Finite-sample and asymptotic radii]\label{thm:calibration}
Assume $p_j^\star\ge\underline p>0$ and $0<\delta<1$, let $M_C=\max_{j,i}C_{ji}$, and define
\begin{equation}\label{eq:rn}
r_{n,\delta,a}
=\frac{n}{n+ma}\sqrt{\frac{m}{2n}\log\frac{2m}{\delta}}
+\frac{a\sqrt{m(m-1)}}{n+ma}.
\end{equation}
If $r_{n,\delta,a}\le\underline p/2$, then with probability at least $1-\delta$,
\begin{equation}\label{eq:fs-radius}
D_\gamma^{\TC}(p^\star\Vert\widetilde p_n)
\le \rho_{n,\delta,\gamma}^{\mathrm{FS}}
:=\frac{\gamma e^{2M_C/\gamma}}{\underline p}
r_{n,\delta,a}^2.
\end{equation}
For fixed $(m,a,\gamma)$,
\begin{equation}\label{eq:asymptotic}
\begin{aligned}
2nD_\gamma^{\TC}(p^\star\Vert\widetilde p_n)
&\ \Rightarrow\ Z\trans H_{\gamma,p^\star}Z,\\[-2pt]
Z&\sim\mathcal N(0,\Sigma^\star),\qquad
\Sigma^\star=\diag(p^\star)-p^\star(p^\star)\trans.
\end{aligned}
\end{equation}
If $c_{1-\delta}(\gamma,p^\star)$ is the corresponding quantile, then
$c_{1-\delta}/(2n)$ has asymptotic coverage $1-\delta$. Replacing
$(p^\star,H,\Sigma^\star)$ by their add-$a$ plug-in estimates is consistent. If the conditional quantile is estimated by Monte Carlo, the number of Gaussian draws must satisfy $B_n\to\infty$.
\end{theorem}

The nonzero eigenvalues of
$(\Sigma^\star)^{1/2}H_{\gamma,p^\star}(\Sigma^\star)^{1/2}$ give the representation
$Z\trans HZ\overset{d}=\sum_{r=1}^{m-1}\lambda_r\chi_{1,r}^2$.
The second-order delta-method mechanism follows \citet{bigot2019central}; the distinctive ingredient here is the explicit forward Hessian \eqref{eq:H} and its use in calibrating a one-sided DRO set.
As $\gamma\to\infty$, $H_{\gamma,p}/\gamma$ approaches the Fisher metric and the normalized limit becomes $\chi^2_{m-1}$; for $C=0$, this statement is exact at every $\gamma$.

\section{Controlled validation}

We test centering, the temperature limits, and the asymptotic plug-in calibration. All reported values were generated by \texttt{anc/reproduce.py} with seed 20260725.

\paragraph{Centering and limiting geometry.}
For $p=(0.70,0.20,0.10)$ on locations $(0,1,2)$ with squared-distance cost, \cref{fig:geometry}a shows the raw-center displacement from \eqref{eq:raw-center}. It grows from $2.1\times10^{-9}$ at $\gamma=0.05$ to $0.704$ at $\gamma=20$, while the tangent-centered minimizer remains exactly $p$. For $q=(0.45,0.35,0.20)$, \cref{fig:geometry}b solves the entropic transport plans and evaluates the exact plan-KL identity: the corrected discrepancy approaches $W_C$ at low temperature and $\gamma\KL(q\Vert p)$ at high temperature, as predicted by \cref{thm:limits}.

\begin{figure}[t]
\centering
\includegraphics[width=\linewidth]{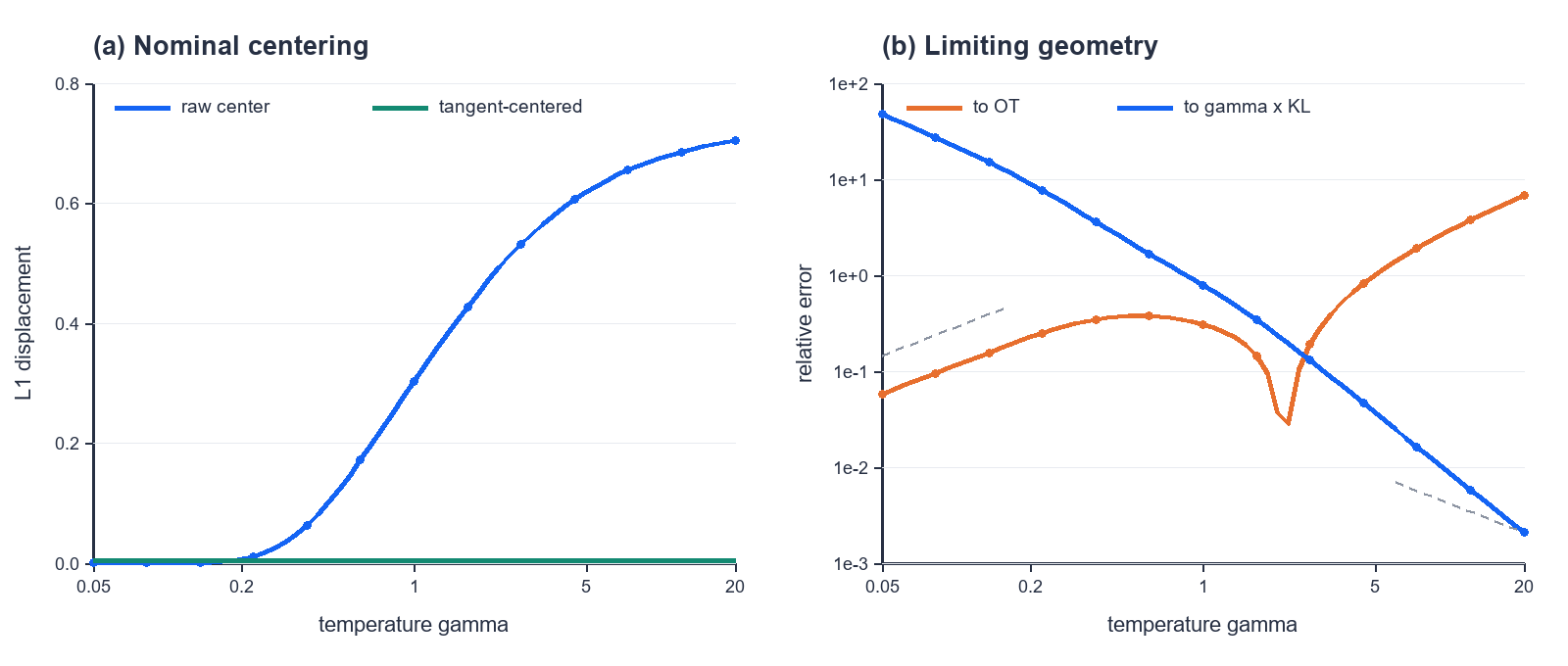}
\caption{Controlled validation. (a) Entropy moves the raw fixed-reference center toward uniform, whereas tangent centering fixes the unique minimizer at $p$. (b) The orange curve is $\abs{D_\gamma^{\TC}-W_C}/W_C$ and the blue curve is $\abs{D_\gamma^{\TC}/\gamma-\KL}/\KL$; dashed lines are slope-$+1$ and slope-$-1$ guides.}
\label{fig:geometry}
\end{figure}

\paragraph{Practical plug-in coverage.}
We use
$p^\star=(0.40,0.25,0.15,0.12,0.08)$,
$x_j=(j-1)/4$, $C_{ji}=(x_j-x_i)^2$, and $a=1/2$.
For each sample, both $H$ and $\Sigma$ are evaluated at $\widetilde p_n$; the $0.95$ generalized chi-square quantile uses a fixed bank of 5,000 Gaussian draws. \Cref{tab:coverage} reports 2,000 independent multinomial repetitions. Coverage is within roughly three Monte Carlo standard errors of $0.95$ in every configuration, including the more curved $\gamma=5$ regime.

\begin{table}[t]
\caption{Plug-in coverage of the nominal $95\%$ radius. Parentheses give Monte Carlo standard errors; the final column is the median data-dependent radius.}
\label{tab:coverage}
\centering
\small
\begin{tabular}{@{}rrrr@{}}
\toprule
$\gamma$ & $n$ & Coverage (s.e.) & Median radius \\
\midrule
0.2 & 50  & 0.9520 (0.0048) & 0.02499 \\
0.2 & 100 & 0.9545 (0.0047) & 0.01244 \\
0.2 & 500 & 0.9470 (0.0050) & 0.00248 \\
1.0 & 50  & 0.9395 (0.0053) & 0.09606 \\
1.0 & 100 & 0.9495 (0.0049) & 0.04803 \\
1.0 & 500 & 0.9565 (0.0046) & 0.00961 \\
5.0 & 50  & 0.9355 (0.0055) & 0.47440 \\
5.0 & 100 & 0.9535 (0.0047) & 0.23720 \\
5.0 & 500 & 0.9475 (0.0050) & 0.04744 \\
\bottomrule
\end{tabular}
\end{table}

\paragraph{DRO primal--dual consistency.}
For $p=(0.55,0.30,0.15)$, $L=(0.10,1.20,-0.40)$, $\gamma=0.7$, and $\rho=0.12$, scalar minimization of \eqref{eq:dual} gives $\lambda^*=1.310730393$. Equation~\eqref{eq:worst-plan} recovers $q^*=(0.395776,0.541719,0.062505)$ with $D_\gamma^{\TC}(q^*\Vert p)=0.119999996$. The primal and dual values are $0.664638662$ and $0.664638668$, an absolute gap of $5.8\times10^{-9}$.

This table evaluates the asymptotic plug-in radius, not the conservative finite-sample radius \eqref{eq:fs-radius}. At $\delta=0.05$, \eqref{eq:rn} gives $r_{n,\delta,a}=0.533,0.377,0.166$ for $n=50,100,500$, all larger than the sufficient threshold $\underline p/2=0.04$. The experiments validate algebra, limits, and practical asymptotic calibration; they do not establish superiority on a downstream benchmark. The method is most relevant when the finite support is decision-relevant and its nominal weights are nonuniform.

\section{Limitations and conclusion}

Tangent centering cleanly separates nominal weights from fixed-reference entropic drift. It is canonical among affine corrections, admits the exact plan-KL representation \eqref{eq:plan-kl}, preserves a scalar Sinkhorn-DRO dual, and connects discrete OT to KL through temperature limits under the stated cost assumptions. Its local metric also supports both conservative and asymptotically exact ambiguity radii.

The guarantees are conditional on a known finite support and fixed cost. Support discretization, learned representations, and cost estimation require additional error terms or sample splitting. Dense evaluation costs $O(m^2)$, while a full Hessian pseudoinverse costs $O(m^3)$; sparse kernels and Hessian-vector methods are needed at large $m$. Finally, tangent centering is deliberately asymmetric. Symmetric Sinkhorn divergences remain preferable for two-sample comparison, whereas DRO has a distinguished nominal anchor.

\subsubsection*{Reproducibility Statement}
All assumptions are stated in the main text, complete proofs are provided in the appendix, and the supplemental script regenerates every numerical value and figure from a fixed seed using only NumPy and Pillow. The appendix records tolerances, repetition counts, and quantile construction.

\clearpage
\bibliography{references}
\bibliographystyle{iclr2026_conference}

\clearpage
\appendix
\section{Convex-analytic preliminaries}\label{app:preliminaries}

This appendix supplies complete proofs and implementation details. All derivatives on the simplex are relative to its affine hull. Adding a multiple of $\one$ to a potential never changes a pairing with $q-p$.

\subsection{Semi-dual and smoothness}

Define $A_{\gamma,p}$ and $s_i(u)$ as in \eqref{eq:A}.

\begin{lemma}[Semi-dual representation]\label{lem:semidual}
For $q,p\in\simplex^\circ$,
\begin{equation}\label{eq:semidual-app}
T_\gamma(q,p)
=\sup_{u\in\R^m}\{u\trans q-A_{\gamma,p}(u)\}
+\gamma\sum_i p_i\log p_i-\gamma.
\end{equation}
Moreover,
\begin{equation}\label{eq:grad-A-app}
\nabla A_{\gamma,p}(u)=\sum_i p_i s_i(u),
\qquad
\nabla^2 A_{\gamma,p}(u)
=\frac1\gamma\sum_i p_i[\diag(s_i(u))-s_i(u)s_i(u)\trans].
\end{equation}
The maximizing potential is unique modulo $\operatorname{span}\{\one\}$.
\end{lemma}

\begin{proof}
Introduce row and column potentials $u,v$ for the two marginal constraints. The scalar identity
\[
\inf_{x\ge0}\left\{\gamma x(\log x-1)+(C_{ji}-u_j-v_i)x\right\}
=-\gamma\exp((u_j+v_i-C_{ji})/\gamma)
\]
gives
\[
T_\gamma(q,p)=\sup_{u,v}
\left\{u\trans q+v\trans p
-\gamma\sum_{j,i}e^{(u_j+v_i-C_{ji})/\gamma}\right\}.
\]
For fixed $u$, first-order optimality in $v_i$ gives
$v_i=\gamma\log p_i-\gamma\log\sum_j e^{(u_j-C_{ji})/\gamma}$.
Substitution proves \eqref{eq:semidual-app}; differentiation gives \eqref{eq:grad-A-app}.

Every $s_i(u)$ has full support. Hence for any nonconstant $z$,
\[
z\trans\nabla^2A_{\gamma,p}(u)z
=\gamma^{-1}\sum_i p_i\operatorname{Var}_{s_i(u)}(z)>0.
\]
The Hessian nullspace is exactly $\operatorname{span}\{\one\}$. After fixing, for example, the gauge $\one\trans u=0$, strict convexity and the usual finite Sinkhorn scaling theorem yield a unique optimizer. Equivalently, the unique positive entropic plan determines both potentials up to the common additive gauge.
\end{proof}

The strict positivity of the Gibbs kernel and $p\in\simplex^\circ$ imply that the self-plan $\Pi_\gamma^{p,p}$ has strictly positive entries. The value $T_\gamma(q,p)$ and its unique optimizer are continuous in $(q,p)$ on the closed product of simplices; this follows directly from compactness of the coupling correspondence, continuity of the objective with $0\log0=0$, and uniqueness of the optimizer. On compact subsets of the interior, the value is $C^\infty$ by \cref{lem:semidual} and the implicit-function theorem on the tangent space.

\section{Proofs for the raw center and tangent geometry}\label{app:geometry-proofs}

\subsection{Proof of \texorpdfstring{\cref{prop:raw-center}}{the raw-center proposition}}

Minimizing over $q$ removes the row-sum constraint, so the problem separates over columns subject only to $\Pi\trans\one=p$. For $p_i>0$, write $\Pi_{ji}=p_i s_j$ with $s\in\simplex$. The column objective is
\[
p_i\sum_j s_j C_{ji}
+\gamma p_i\left(\log p_i-1+\sum_j s_j\log s_j\right).
\]
Its unique minimizer is $s_j=K_{ji}/Z_i$. (A column with $p_i=0$ is identically zero.) Averaging the resulting column probabilities with weights $p_i$ proves \eqref{eq:raw-center}. The minimum value, useful for checking implementations, is
\begin{equation}\label{eq:raw-min-value-app}
\min_qT_\gamma(q,p)
=\gamma\sum_i p_i(\log p_i-1-\log Z_i),
\end{equation}
with the convention $0\log0=0$. Since $K_{ji}\to1$ as $\gamma\to\infty$, every normalized column converges to $m^{-1}\one$.

For the low-temperature bound, the mass that the $i$th Gibbs column assigns away from $i$ is at most
$\sum_{j\ne i}e^{-C_{ji}/\gamma}\le(m-1)e^{-\Delta_C/\gamma}$.
The $\ell_1$ distance between that column and the $i$th point mass is twice its off-diagonal mass. Averaging over $p_i$ proves the claim.

\subsection{Proof of \texorpdfstring{\cref{prop:canonical}}{the canonical-centering proposition}}

Because $p$ lies in the relative interior, it minimizes the differentiable convex function $D_a(\cdot,p)$ on $\simplex$ if and only if its relative directional derivative vanishes in every $v\in\cT$. That condition is
$(\nu_{\gamma,p}-a)\trans v=0$ for every $v\in\cT$, equivalently
$a-\nu_{\gamma,p}\in\cT^\perp=\operatorname{span}\{\one\}$.
Strict convexity of $F_{\gamma,p}$ then makes the minimizer unique. The constant term is already forced by the normalization $D_a(p,p)=0$.

\subsection{Proof of \texorpdfstring{\cref{thm:plan-kl}}{the plan-KL theorem}}

First let $q\in\simplex^\circ$, and let $u_q$ and $\nu=\nu_{\gamma,p}$ be semi-dual potentials satisfying
$\nabla A_{\gamma,p}(u_q)=q$ and $\nabla A_{\gamma,p}(\nu)=p$.
Up to the constant in \eqref{eq:semidual-app}, $F_{\gamma,p}=A_{\gamma,p}^*$. The standard conjugate-Bregman identity gives
\begin{equation}\label{eq:conjugate-bregman-app}
D_{F_{\gamma,p}}(q,p)=D_{A_{\gamma,p}}(\nu,u_q).
\end{equation}
For a single column, direct expansion of the log-partition function gives
\begin{align*}
&\gamma\log\sum_j e^{(\nu_j-C_{ji})/\gamma}
-\gamma\log\sum_j e^{(u_{q,j}-C_{ji})/\gamma}
-s_i(u_q)\trans(\nu-u_q)\\
&\hspace{1.1in}=\gamma\KL(s_i(u_q)\Vert s_i(\nu)).
\end{align*}
Multiplying by $p_i$ and summing over columns shows
\begin{equation}\label{eq:conditional-plan-kl-app}
D_\gamma^{\TC}(q\Vert p)
=\gamma\sum_i p_i\KL(s_i(u_q)\Vert s_i(\nu)).
\end{equation}
The optimal plans satisfy
$\Pi_{\gamma,ji}^{q,p}=p_i s_{ji}(u_q)$ and
$\Pi_{\gamma,ji}^{p,p}=p_i s_{ji}(\nu)$, so \eqref{eq:conditional-plan-kl-app} is exactly \eqref{eq:plan-kl}.

The identical calculation with any $r\in\simplex^\circ$ and its semi-dual potential $u_r$ in place of $(p,\nu)$ gives
\[
B_{F_{\gamma,p}}(q,r)
=\gamma\KL(\Pi_\gamma^{q,p}\Vert\Pi_\gamma^{r,p}).
\]
Data processing under the row-marginal map gives the first inequality in \eqref{eq:global-plan-kl}; Pinsker gives the second. Hence $F_{\gamma,p}$, and therefore its affine shift $D_\gamma^{\TC}(\cdot\Vert p)$, is $\gamma$-strongly convex with respect to $\ell_1$ in the Bregman sense.

For boundary $q$, take $q_t=(1-t)q+tp$ and let $t\downarrow0$. Continuity of the entropic value and optimizer gives
$\Pi_\gamma^{q_t,p}\to\Pi_\gamma^{q,p}$. Since the reference self-plan is strictly positive and $x\mapsto x\log x$ is continuous at zero, both sides of \eqref{eq:plan-kl} converge, proving the identity on the closed simplex.

Apply the KL chain rule to the row marginal map $(j,i)\mapsto j$:
\[
\KL(\Pi_\gamma^{q,p}\Vert\Pi_\gamma^{p,p})
=\KL(q\Vert p)
+\sum_{j:q_j>0}q_j
\KL(\Pi_\gamma^{q,p}(\cdot\mid j)
\Vert\Pi_\gamma^{p,p}(\cdot\mid j)).
\]
This proves \eqref{eq:chain-kl} and the first inequality in \eqref{eq:kl-tv-lower}; Pinsker's inequality proves the second. Convexity follows directly from the defining Bregman remainder in its first argument. Continuity and compactness of the simplex give compact sublevel sets, and strict positivity of KL away from equality gives exact centering.

\subsection{Proof of \texorpdfstring{\cref{prop:reference-invariance}}{the reference-invariance proposition} and arbitrary fixed references}

For a fixed $R\in\R_{++}^{m\times m}$, define
\begin{equation}\label{eq:reference-transport-app}
T_{\gamma,R}(q,p)=\min_{\Pi\in\Gamma(q,p)}
\left\{\inner{C}{\Pi}
+\gamma\sum_{j,i}\Pi_{ji}\left(\log\frac{\Pi_{ji}}{R_{ji}}-1\right)\right\}.
\end{equation}
If $R'_{ji}=r_jR_{ji}c_i$, then
\begin{equation}\label{eq:reference-affine-app}
T_{\gamma,R'}(q,p)
=T_{\gamma,R}(q,p)
-\gamma\sum_jq_j\log r_j
-\gamma\sum_i p_i\log c_i.
\end{equation}
The first added term is affine in the varying marginal and the second is constant. Both disappear from the anchored Bregman remainder. Adding a cost $a_j+b_i$ is equivalent to multiplying $R$ by $e^{-a_j/\gamma}e^{-b_i/\gamma}$ and gives the same conclusion.

The semi-dual for \eqref{eq:reference-transport-app} is obtained by replacing each exponential in \eqref{eq:A} by $R_{ji}$ times that exponential. The proof of \cref{thm:plan-kl} is unchanged, so the exact plan-KL identity holds for every fixed positive $R$. The high-temperature KL limit in \cref{thm:limits}, however, uses the counting/separable-reference structure and should not be transferred verbatim to an arbitrary nonseparable $R$.

\section{Proof of the robust dual}\label{app:dual-proof}

With the normalized gauge $p\trans\nu=T_\gamma(p,p)$ and $\bar C_{ji}=C_{ji}-\nu_j$, the constraint
$D_\gamma^{\TC}(q\Vert p)\le\rho$ is equivalent to the existence of $\Pi\trans\one=p$ such that
\begin{equation}\label{eq:coupling-budget-app}
\inner{\bar C}{\Pi}+\gamma\sum_{j,i}\Pi_{ji}(\log\Pi_{ji}-1)\le\rho.
\end{equation}
Indeed, minimizing the left side over plans with first marginal $q$ gives
$T_\gamma(q,p)-\nu\trans q$, while the gauge makes the right-side offset
$T_\gamma(p,p)-\nu\trans p$ equal to zero.

Attach $\lambda\ge0$ to \eqref{eq:coupling-budget-app} and $\beta_i$ to each column constraint. The Lagrangian for the maximization problem is
\begin{align*}
\mathcal L(\Pi,\lambda,\beta)
={}&\lambda\rho+p\trans\beta\\
&+\sum_{j,i}\left[\Pi_{ji}(L_j-\beta_i-\lambda\bar C_{ji})
-\lambda\gamma\Pi_{ji}(\log\Pi_{ji}-1)\right].
\end{align*}
For $\lambda>0$,
\[
\sup_{x\ge0}\{ax-\lambda\gamma x(\log x-1)\}
=\lambda\gamma e^{a/(\lambda\gamma)}.
\]
Therefore the dual before eliminating $\beta$ is
\begin{equation}\label{eq:dual-beta-app}
\inf_{\lambda\ge0,\,\beta\in\R^m}
\left\{\lambda\rho+p\trans\beta
+\lambda\gamma\sum_{j,i}
e^{(L_j-\beta_i-\lambda\bar C_{ji})/(\lambda\gamma)}\right\},
\end{equation}
with the lower-semicontinuous boundary interpretation at $\lambda=0$. The nominal self-plan is strictly positive and has budget zero, so $\rho>0$ gives Slater slack. Strong duality follows.

For fixed $(\theta,\lambda)$, set
$S_i=\sum_j e^{(L_j-\lambda\bar C_{ji})/(\lambda\gamma)}$.
The terms involving $\beta_i$ are
$p_i\beta_i+\lambda\gamma S_i e^{-\beta_i/(\lambda\gamma)}$.
Their minimizer is $\beta_i^*=\lambda\gamma\log(S_i/p_i)$ and their minimum is
$\lambda\gamma p_i[1+\log(S_i/p_i)]$.
Summing proves \eqref{eq:dual}.

At a positive optimal multiplier, maximizing each scalar Lagrangian term gives
$\Pi_{ji}^*=e^{(L_j-\beta_i^*-\lambda^*\bar C_{ji})/(\lambda^*\gamma)}$.
Substituting $\beta_i^*$ proves \eqref{eq:worst-plan}; strict concavity in $\Pi$ gives uniqueness. The exponential perspective in \eqref{eq:dual-beta-app} is jointly convex. Composing its increasing argument with convex $L_j(\theta)$ and partially minimizing over $\beta$ preserves convexity in $(\theta,\lambda)$. The gradient statement follows from Danskin's theorem; in the nonunique case, under convex losses, every maximizing distribution gives a convex subgradient.

For the fixed reference $R$ in \eqref{eq:reference-transport-app}, the same derivation inserts $R_{ji}$ in every exponential and in the numerator and denominator of \eqref{eq:worst-plan}.

\section{Proofs of the temperature limits}\label{app:limit-proofs}

\subsection{Low temperature}

Every coupling has total mass one on $m^2$ entries, so
$-2\log m\le\sum_{j,i}\Pi_{ji}\log\Pi_{ji}\le0$. Hence
\begin{equation}\label{eq:T-W-basic-app}
W_C(q,p)-\gamma(1+2\log m)
\le T_\gamma(q,p)\le W_C(q,p)-\gamma.
\end{equation}

\begin{lemma}[Bounded self-potential]\label{lem:self-potential-app}
Under the low-temperature assumptions of \cref{thm:limits}, the self-plan has a symmetric representation
\[
\Pi_{\gamma,ji}^{p,p}
=\exp((u_{\gamma,j}+u_{\gamma,i}-C_{ji})/\gamma)
\]
with
\begin{equation}\label{eq:self-potential-bound-app}
\gamma(\log p_j-\tfrac12\log m)
\le u_{\gamma,j}\le\tfrac\gamma2\log p_j,
\qquad
\norm{u_\gamma}_\infty
\le\gamma\log(\sqrt m/p_{\min}).
\end{equation}
\end{lemma}

\begin{proof}
Symmetry and uniqueness of the self-plan allow equal row and column potentials. Put
$a_j=e^{u_{\gamma,j}/\gamma}$. Then
$\Pi_{ji}=a_ja_i e^{-C_{ji}/\gamma}$ and
$p_j=a_j\sum_i a_i e^{-C_{ji}/\gamma}$.
The diagonal term gives $p_j\ge a_j^2$, so $a_j\le\sqrt{p_j}$. Conversely,
\[
p_j\le a_j\sum_i a_i
\le a_j\sum_i\sqrt{p_i}\le a_j\sqrt m.
\]
Thus $p_j/\sqrt m\le a_j\le\sqrt{p_j}$. Taking logarithms gives the result.
\end{proof}

The row self-potential is a valid representative of $\nu_{\gamma,p}$. Combining \eqref{eq:T-W-basic-app}, $W_C(p,p)=0$, \eqref{eq:self-potential-bound-app}, and $\norm{q-p}_1\le2$ gives
\begin{align*}
\abs{D_\gamma^{\TC}(q\Vert p)-W_C(q,p)}
&\le \abs{T_\gamma(q,p)-W_C(q,p)}+\abs{T_\gamma(p,p)}
+\abs{u_\gamma\trans(q-p)}\\
&\le2\gamma[1+2\log m+\log(\sqrt m/p_{\min})],
\end{align*}
which proves \eqref{eq:low-bound}.

For completeness, let $V_W(r;L)=\sup_{W_C(q,p)\le r}q\trans L$. Uniform closeness of the two discrepancies gives
\[
\{W_C\le\rho-\varepsilon_\gamma\}
\subseteq\cU_{\rho,\gamma}^{\TC}(p)
\subseteq\{W_C\le\rho+\varepsilon_\gamma\}.
\]
If $\norm{L}_\infty\le B$, contract any $q$ feasible at radius $\rho+s$ toward $p$ by the factor $\rho/(\rho+s)$. Convexity of $W_C(\cdot,p)$ makes the contraction feasible at radius $\rho$, while its objective changes by at most $2Bs/\rho$. The analogous contraction from $\rho$ to $\rho-s$ gives
$\abs{V_W(\rho\pm s;L)-V_W(\rho;L)}\le2Bs/\rho$ for $s<\rho$. Taking $s=\varepsilon_\gamma$ proves \eqref{eq:risk-low-bound}.

\subsection{High temperature}

Let $h(q)=\sum_jq_j\log q_j$. For any $\Pi\in\Gamma(q,p)$,
\begin{equation}\label{eq:entropy-decomp-app}
\sum_{j,i}\Pi_{ji}(\log\Pi_{ji}-1)
=h(q)+h(p)-1+\KL(\Pi\Vert qp\trans).
\end{equation}
Therefore
\begin{equation}\label{eq:T-high-decomp-app}
T_\gamma(q,p)=\gamma[h(q)+h(p)-1]+R_\gamma(q,p),
\end{equation}
where
\[
R_\gamma(q,p)=\min_{\Pi\in\Gamma(q,p)}
\{\inner{C}{\Pi}+\gamma\KL(\Pi\Vert qp\trans)\}.
\]
Writing $M_C=\max_{j,i}C_{ji}$, the independent coupling gives the upper bound
$R_\gamma(q,p)\le q\trans Cp$. For any feasible $\Pi$, Pinsker gives
\[
\inner{C}{\Pi}+\gamma\KL(\Pi\Vert qp\trans)
\ge q\trans Cp-M_Ct+\tfrac\gamma2t^2,
\quad t=\norm{\Pi-qp\trans}_1.
\]
Minimizing the right side over $t\ge0$ yields
\begin{equation}\label{eq:R-bound-app}
q\trans Cp-\frac{M_C^2}{2\gamma}
\le R_\gamma(q,p)\le q\trans Cp
\end{equation}
uniformly on the closed simplex.

Set $f_\gamma(q)=\gamma^{-1}T_\gamma(q,p)$. Equations
\eqref{eq:T-high-decomp-app}--\eqref{eq:R-bound-app} give
\begin{equation}\label{eq:f-high-uniform-app}
f_\gamma(q)=h(q)+h(p)-1+\gamma^{-1}q\trans Cp+O(\gamma^{-2})
\end{equation}
uniformly in $q$. In particular, $f_\gamma\to f_\infty:=h(\cdot)+h(p)-1$ uniformly. We next justify convergence of normalized relative gradients at the fixed interior point $p$. For $v\in\cT$, choose $t>0$ with $p\pm tv\in\simplex^\circ$. Convexity gives
\[
\frac{f_\gamma(p)-f_\gamma(p-tv)}{t}
\le \nabla f_\gamma(p)\trans v
\le\frac{f_\gamma(p+tv)-f_\gamma(p)}{t}.
\]
First let $\gamma\to\infty$ using uniform convergence, then let $t\downarrow0$. The two bounds meet at
$\nabla h(p)\trans v$. Finite dimensionality implies convergence of the normalized relative gradients on $\cT$. Hence
\begin{align*}
\sup_{q\in\simplex}\Bigl|
&[f_\gamma(q)-f_\gamma(p)-\nabla f_\gamma(p)\trans(q-p)]\\
&-[h(q)-h(p)-\nabla h(p)\trans(q-p)]\Bigr|
\longrightarrow0.
\end{align*}
The second bracket is $\KL(q\Vert p)$, including at boundary points by continuity. This proves \eqref{eq:high-limit}.

Let $g_\gamma=D_\gamma^{\TC}/\gamma$ and $g=\KL(\cdot\Vert p)$. Uniform convergence on compact $\simplex$ implies that any limit point of
$\{g_\gamma\le\tau\}$ lies in $\{g\le\tau\}$. Conversely, if $g(q)<\tau$, then $q$ is eventually feasible. If $g(q)=\tau>0$, use
$q_t=(1-t)q+tp$: strict convexity along this segment gives
$g(q_t)<\tau$, and one may choose $t=t_\gamma\downarrow0$ slowly enough that $q_{t_\gamma}$ is feasible for $g_\gamma$. Compactness turns these two facts into Hausdorff convergence. Uniform convergence of bounded linear support functions follows by the same compactness argument.

\section{Proofs for local geometry and calibration}\label{app:calibration-proofs}

\subsection{Proof of \texorpdfstring{\cref{thm:hessian}}{the local-metric theorem}}

By \eqref{eq:grad-A-app},
$\nabla^2A_{\gamma,p}(\nu)=\gamma^{-1}M_{\gamma,p}$.
For nonzero $v\in\cT$, every $s_i(\nu)$ has full support and
\[
v\trans M_{\gamma,p}v
=\sum_i p_i\operatorname{Var}_{s_i(\nu)}(v)>0.
\]
Since $T_\gamma(\cdot,p)$ is $A_{\gamma,p}^*$ up to a constant, their intrinsic Hessians are inverses, proving
$H_{\gamma,p}=\gamma M_{\gamma,p}^\dagger$ and the Taylor expansion.

There are two equivalent proofs of \eqref{eq:fisher-lower}. Differentiating the global inequality
$D_\gamma^{\TC}(q\Vert p)\ge\gamma\KL(q\Vert p)$ twice at $q=p$ gives it immediately. Alternatively, under the self-plan let $J$ be the row state and $I$ the column state. Both marginals are $p$, and
$M_{\gamma,p}=\E[\operatorname{Cov}(e_J\mid I)]$. The law of total covariance gives
\[
M_{\gamma,p}\preceq\operatorname{Cov}(e_J)
=\diag(p)-pp\trans.
\]
Inversion on $\cT$ reverses the Loewner order, yielding
$v\trans H_{\gamma,p}v\ge\gamma\sum_jv_j^2/p_j$.

\subsection{An explicit curvature bound}

\begin{lemma}[Uniform interior curvature]\label{lem:explicit-curvature}
Let $\simplex(\epsilon)=\{q\in\simplex:q_j\ge\epsilon\}$ and $M_C=\max C_{ji}$. For $q,p\in\simplex(\epsilon)$,
\begin{equation}\label{eq:curvature-explicit-app}
\norm{\nabla_{11}^2T_\gamma(q,p)}_{\mathrm{op},\cT}
\le\frac{\gamma e^{2M_C/\gamma}}{\epsilon}.
\end{equation}
\end{lemma}

\begin{proof}
Let $K_{ji}=e^{-C_{ji}/\gamma}$ and $\kappa=e^{-M_C/\gamma}$, so
$\kappa\le K_{ji}\le1$. Write the positive optimal plan as
$\Pi=\diag(a)K\diag(b)$ and set $A=\sum_ja_j$, $B=\sum_i b_i$. Total mass gives
$1=\sum_{j,i}a_jK_{ji}b_i\ge\kappa AB$, hence $AB\le\kappa^{-1}$. The row constraint gives
$q_j=a_j\sum_iK_{ji}b_i\le a_jB$, so $a_j\ge\epsilon/B$. Therefore every conditional column probability satisfies
\[
s_{ji}=\frac{a_jK_{ji}}{\sum_ka_kK_{ki}}
\ge\frac{(\epsilon/B)\kappa}{A}
\ge\epsilon\kappa^2.
\]
For $v\in\cT$,
\[
\operatorname{Var}_{s_i}(v)
=\min_c\sum_j s_{ji}(v_j-c)^2
\ge\epsilon\kappa^2\min_c\sum_j(v_j-c)^2
=\epsilon\kappa^2\norm{v}_2^2.
\]
Thus the log-partition covariance is bounded below by
$\epsilon\kappa^2I$ on $\cT$. Inverting it in the Hessian formula proves the claim.
\end{proof}

\subsection{Proof of the finite-sample statement in \texorpdfstring{\cref{thm:calibration}}{the calibration theorem}}

Let $\widehat p_n=N/n$. A coordinatewise Hoeffding bound and union bound give, with probability at least $1-\delta$,
\begin{equation}\label{eq:hoeffding-app}
\norm{\widehat p_n-p^\star}_2
\le\sqrt{\frac{m}{2n}\log\frac{2m}{\delta}}.
\end{equation}
The smoothed estimator obeys
\[
\widetilde p_n-p^\star
=\frac{n}{n+ma}(\widehat p_n-p^\star)
+\frac{a}{n+ma}(\one-mp^\star).
\]
The maximum of $\norm{\one-mp}_2$ over the simplex is attained at a vertex and equals $\sqrt{m(m-1)}$. Thus \eqref{eq:hoeffding-app} implies
$\norm{\widetilde p_n-p^\star}_2\le r_{n,\delta,a}$.
If $r_{n,\delta,a}\le\underline p/2$, the entire segment between $p^\star$ and $\widetilde p_n$ belongs to
$\simplex(\underline p/2)$.

For fixed second marginal $\widetilde p_n$, the discrepancy is the Taylor remainder of
$q\mapsto T_\gamma(q,\widetilde p_n)$ about $q=\widetilde p_n$. By \cref{lem:explicit-curvature},
\[
D_\gamma^{\TC}(p^\star\Vert\widetilde p_n)
\le\frac12\frac{\gamma e^{2M_C/\gamma}}{\underline p/2}
\norm{p^\star-\widetilde p_n}_2^2
\le\frac{\gamma e^{2M_C/\gamma}}{\underline p}r_{n,\delta,a}^2.
\]
This proves \eqref{eq:fs-radius}.

\subsection{Proof of the generalized chi-square limit}

The add-$a$ bias is $O(n^{-1})$, so the multinomial central limit theorem gives
\[
\sqrt n(\widetilde p_n-p^\star)\Rightarrow
Z\sim\mathcal N(0,\Sigma^\star).
\]
For each realized interior nominal $r$, $D_\gamma^{\TC}(q\Vert r)$ is the Bregman remainder in $q$ about $r$. Entropic OT is $C^\infty$ on a common interior neighborhood of $p^\star$, so Taylor expansion uniformly for $r\to p^\star$ gives
\begin{equation}\label{eq:random-taylor-app}
D_\gamma^{\TC}(p^\star\Vert\widetilde p_n)
=\frac12(p^\star-\widetilde p_n)\trans
H_{\gamma,\widetilde p_n}(p^\star-\widetilde p_n)
+o_p(n^{-1}).
\end{equation}
Continuity of $H_{\gamma,p}$ and Slutsky's theorem prove \eqref{eq:asymptotic}.

Let $\lambda_1,\ldots,\lambda_{m-1}$ be the positive eigenvalues of
$(\Sigma^\star)^{1/2}H_{\gamma,p^\star}(\Sigma^\star)^{1/2}$. Diagonalizing this matrix gives
$Z\trans HZ\overset d=\sum_{r=1}^{m-1}\lambda_r\chi_{1,r}^2$.
The distribution has a continuous positive density away from zero, so its positive quantiles are continuity points. The plug-in matrices converge in probability, their nonzero eigenvalues converge, and so do the associated quantiles. If the quantile is estimated from $B_n$ conditional Gaussian draws, conditional Glivenko--Cantelli (or the Dvoretzky--Kiefer--Wolfowitz inequality) makes the empirical quantile error vanish provided $B_n\to\infty$. Combining this fact with \eqref{eq:random-taylor-app} proves plug-in coverage.

Finally, as $\gamma\to\infty$, the self-plan converges to $pp\trans$, so every conditional vector $s_i(\nu_{\gamma,p})$ converges to $p$. Hence
$M_{\gamma,p}\to\diag(p)-pp\trans$ and
\begin{equation}\label{eq:high-hessian-app}
\gamma^{-1}H_{\gamma,p}\to(\diag(p)-pp\trans)^\dagger
\quad\text{on }\cT.
\end{equation}
The resulting normalized quadratic form is $\chi^2_{m-1}$. When $C=0$, the optimal plan is $qp\trans$, so
$D_\gamma^{\TC}=\gamma\KL(q\Vert p)$ and \eqref{eq:high-hessian-app} is exact for every $\gamma$.

\section{Relation to earlier EOT Bregman orientations}\label{app:prior-bregman}

For completeness, we verify \eqref{eq:product-decomposition} and make the orientation distinction algebraic. Let
\[
h_{\gamma,p}(q)
=\min_{\Pi\in\Gamma(q,p)}
\left\{\inner{C}{\Pi}
+\gamma\KL(\Pi\Vert qp\trans)\right\}.
\]
For any $q,p\in\simplex^\circ$ and $\Pi\in\Gamma(q,p)$,
\[
\KL(\Pi\Vert qp\trans)
=\sum_{j,i}\Pi_{ji}\log\Pi_{ji}
-\sum_jq_j\log q_j-\sum_ip_i\log p_i.
\]
Consequently,
\begin{equation}\label{eq:T-h-relation-app}
T_\gamma(q,p)
=h_{\gamma,p}(q)
+\gamma\sum_jq_j\log q_j
+\gamma\sum_ip_i\log p_i-\gamma.
\end{equation}
Taking the forward Bregman remainder of both sides in the first argument, anchored at $p$, proves
\[
D_\gamma^{\TC}(q\Vert p)
=B_{h_{\gamma,p}}(q,p)+\gamma\KL(q\Vert p).
\]
Combining this equality with \eqref{eq:chain-kl} further yields
\[
B_{h_{\gamma,p}}(q,p)
=\gamma\sum_{j:q_j>0}q_j
\KL\!\left(\Pi_\gamma^{q,p}(\cdot\mid j)
\middle\Vert\Pi_\gamma^{p,p}(\cdot\mid j)\right).
\]
The boundary extension follows by the same interior approximation used in the proof of \cref{thm:plan-kl}.

The product-reference Bregman construction appears in \citet{kerimkulov2025mirror}. \citet[Eqs.~49--50]{amari2018information} explicitly introduce both KL orientations between referenced optimal plans, call the second the dual candidate, and then analyze the self-plan-first orientation obtained from their Eq.~49. For symmetric $C$, setting their reference to $p$, transposing the plans to match our row-$q$/column-$p$ convention, and matching regularization scales turns their Eq.~50 exactly into \eqref{eq:plan-kl}. Thus the forward plan-KL candidate is prior art. The results here identify it with the unique affine tangent correction of fixed/counting-reference EOT and derive its conditional/marginal decomposition, DRO reformulation, limits, and calibration. KL asymmetry still matters: $\gamma^{-1}D_\gamma^{\TC}(q\Vert p)\to\KL(q\Vert p)$ and, for symmetric zero-diagonal costs, $D_\gamma^{\TC}(q\Vert p)\to W_C(q,p)$ without rescaling as $\gamma\downarrow0$.

\section{Relation to symmetric Sinkhorn debiasing}\label{app:symmetric}

When $C$ is symmetric, define
\[
S_\gamma(q,p)=T_\gamma(q,p)-\tfrac12T_\gamma(q,q)-\tfrac12T_\gamma(p,p),
\qquad G_\gamma(q)=-\tfrac12T_\gamma(q,q).
\]
Symmetry gives equal first- and second-marginal derivatives on the diagonal, hence
$\nabla G_\gamma(p)=-\nu_{\gamma,p}$ in the relative tangent space. Direct algebra yields
\begin{equation}\label{eq:tc-symmetric-relation-app}
D_\gamma^{\TC}(q\Vert p)
=S_\gamma(q,p)
-\left[G_\gamma(q)-G_\gamma(p)-\nabla G_\gamma(p)\trans(q-p)\right].
\end{equation}
Thus symmetric self-cost subtraction and one-sided tangent subtraction are not interchangeable. The former targets symmetric comparison and an OT-to-kernel limit under additional assumptions; the latter preserves convexity inherited from $T_\gamma(\cdot,p)$ and produces the plan-KL and marginal-KL geometry proved above.

\section{Implementation and reproduction details}\label{app:reproduction}

\subsection{Stable computation}

For production use, the self-plan can be computed by alternating scaling
\[
a\leftarrow p/(Kb),\qquad b\leftarrow p/(K\trans a),
\qquad K_{ji}=e^{-C_{ji}/\gamma},
\]
with log-domain stabilization when the dynamic range of $C/\gamma$ is large. A row potential is $\nu=\gamma\log a$; it is shifted by a constant until
$p\trans\nu=T_\gamma(p,p)$. For fixed $(\theta,\lambda)$ in \eqref{eq:dual}, set
\[
x_{ji}=\frac{L_j(\theta)-\lambda\bar C_{ji}}{\lambda\gamma},
\qquad m_i=\max_jx_{ji}.
\]
Then compute
$\log\sum_je^{x_{ji}}=m_i+\log\sum_je^{x_{ji}-m_i}$ and normalize the same shifted exponentials to obtain \eqref{eq:worst-plan}.

For reference, differentiating the scalar dual objective $\Psi(\theta,\lambda)$ gives
\begin{equation}\label{eq:lambda-derivative-app}
\partial_\lambda\Psi
=\rho+\gamma
+\gamma\sum_i p_i\log\!\left[\frac1{p_i}\sum_je^{x_{ji}}\right]
-\frac1\lambda\sum_i p_i\sum_j\omega_{ji}L_j(\theta),
\end{equation}
where $\omega_{ji}=e^{x_{ji}}/\sum_ke^{x_{ki}}$. The corrected cost cancels from the last expectation after differentiating the scaled exponent. Searching over $\log\lambda$ enforces positivity.

The local covariance can be read directly from the self-plan:
\[
s_i(\nu)=\Pi_{\gamma,\cdot i}^{p,p}/p_i,
\qquad
M_{\gamma,p}=\sum_i p_i[\diag(s_i)-s_is_i\trans].
\]
Thus no differentiation through the Sinkhorn iterations is required.

\subsection{Controlled simulations}

The supplied script \texttt{anc/reproduce.py} uses NumPy for all transport, linear algebra, random generation, and quantiles, and Pillow to draw the figure. It performs the following checks.
\begin{enumerate}[leftmargin=1.5em,itemsep=2pt]
\item It evaluates the raw-center formula \eqref{eq:raw-center} on a logarithmic temperature grid.
\item It computes $D_\gamma^{\TC}$ from the plan-KL identity and computes the one-dimensional unregularized OT cost by the monotone greedy coupling. This produces \cref{fig:geometry}.
\item It minimizes the scalar objective \eqref{eq:dual}, recovers the conditional worst-case plan from \eqref{eq:worst-plan}, and reports its budget residual and primal--dual objective gap.
\item For coverage, it draws one fixed $5000\times5$ bank of squared standard Gaussians. For each of 2,000 fresh multinomial samples per $(n,\gamma)$, it recomputes the add-$1/2$ nominal, self-plan, plug-in Hessian, covariance eigenvalues, conditional $0.95$ quantile, and containment event. Reusing the Gaussian bank is a common-random-number device; the experiment assesses a finite-$B$ approximation, while the theorem requires $B_n\to\infty$ for exact asymptotics.
\item Alternating-scaling marginal residual tolerance is $2\times10^{-13}$, checked every 20 updates, with a maximum of 10,000 iterations. For nearly diagonal symmetric self-plans in the low-temperature figure, the script instead solves the symmetric scaling equations by damped Newton iteration at the same tolerance, with a 1,000-iteration cap. The largest exponent range used is $\max C/\gamma=80$, so direct exponentiation is safe in these experiments; production code should use the log-domain scheme above for more extreme inputs.
\end{enumerate}

\section{Statement on the use of language models}\label{app:llm}

An OpenAI language model assisted with language editing, formatting, manuscript organization, mathematical consistency checks, reference verification, and development of the reproducibility code. The author independently verified all mathematical arguments, citations, and numerical results and takes full responsibility for the manuscript.

\end{document}